\documentclass[12pt]{article}

\usepackage{color}

\usepackage{eucal}

\usepackage[all]{xy}

\usepackage{amssymb, amsmath, amsthm, eucal}

\usepackage{mathptmx}
\usepackage[scaled=.90]{helvet}
\usepackage{courier}

\usepackage[pdftex]{hyperref}

\hypersetup{
  colorlinks = true,
  linkcolor = black,
  urlcolor = blue, 
  citecolor = black,
}

\newcommand{\bbF}{\mathbb{F}}
\newcommand{\bbN}{\mathbb{N}}

\newcommand{\calU}{\mathcal{U}}
\newcommand{\calV}{\mathcal{V}}
\newcommand{\calW}{\mathcal{W}}

\newcommand{\rmA}{\mathrm{A}}
\newcommand{\rmB}{\mathrm{B}}
\newcommand{\rmE}{\mathrm{E}}
\newcommand{\rmF}{\mathrm{F}}
\newcommand{\rmK}{\mathrm{K}}
\newcommand{\rmL}{\mathrm{L}}
\newcommand{\rmN}{\mathrm{N}}
\newcommand{\rmP}{\mathrm{P}}
\newcommand{\rmS}{\mathrm{S}}
\newcommand{\rmY}{\mathrm{Y}}

\newcommand{\bfB}{\mathbf{B}}
\newcommand{\bfC}{\mathbf{C}}
\newcommand{\bfD}{\mathbf{D}}
\newcommand{\bfG}{\mathbf{G}}
\newcommand{\bfL}{\mathbf{L}}
\newcommand{\bfM}{\mathbf{M}}
\newcommand{\bfN}{\mathbf{N}}
\newcommand{\bfP}{\mathbf{P}}
\newcommand{\bfQ}{\mathbf{Q}}
\newcommand{\bfS}{\mathbf{S}}
\newcommand{\bfT}{\mathbf{T}}
\newcommand{\bfW}{\mathbf{W}}

\newcommand{\fib}{\mathrm{fib}}
\newcommand{\fc}{\mathrm{fc}}
\newcommand{\op}{\mathrm{op}}

\DeclareMathOperator{\Ho}{Ho}
\DeclareMathOperator{\id}{id}
\DeclareMathOperator{\Id}{Id}

\DeclareMathOperator{\Tot}{Tot}
\DeclareMathOperator{\Aut}{Aut}

\newcommand{\h}{\mathrm{h}}

\DeclareMathOperator{\Z}{Z}
\newcommand{\hoZ}{\Z^\h}

\DeclareMathOperator{\Map}{Map}

\DeclareMathOperator{\Nat}{Nat}
\newcommand{\hoNat}{\Nat^\h}

\DeclareMathOperator{\LKan}{LKan}
\newcommand{\hoLKan}{\LKan^\h}

\DeclareMathOperator{\Set}{Set}
\newcommand{\sSet}{\Set_\Delta}

\DeclareMathOperator{\Pre}{Pre}
\newcommand{\sPre}{\Pre_\Delta}

\DeclareMathOperator{\Alg}{Alg}
\newcommand{\sAlg}{\Alg_\Delta}
\newcommand{\hAlg}{\Alg_\Delta^\h}

\newtheorem{theorem}{Theorem}
\newtheorem{proposition}[theorem]{Proposition}
\newtheorem{corollary}[theorem]{Corollary}

\theoremstyle{definition}
\newtheorem{definition}[theorem]{Definition}
\newtheorem{example}[theorem]{Example}
\newtheorem{examples}[theorem]{Examples}
\newtheorem{remark}[theorem]{Remark}

\numberwithin{theorem}{section}
\numberwithin{equation}{section}

\title{Frobenius and the derived centers\\ of algebraic theories}

\author{William G. Dwyer and Markus Szymik}

\date{October 2014}

\begin{document}

\maketitle

\renewcommand{\abstractname}{}

\begin{abstract}
\noindent We show that the derived center of the category of simplicial algebras over every algebraic theory is homotopically discrete, with the abelian monoid of components isomorphic to the center of the category of discrete algebras. For example, in the case of commutative algebras in characteristic~$p$, this center is freely generated by Frobenius. Our proof involves the calculation of homotopy coherent centers of categories of simplicial presheaves as well as of Bousfield localizations. Numerous other classes of examples are discussed.
\end{abstract}


\section{Introduction}

Algebra in prime characteristic~$p$ comes with a surprise: For each commutative ring~$A$ such that~$p=0$ in~$A$, the~$p$-th power map~$a\mapsto a^p$ is not only multiplicative, but also additive. This defines Frobenius~$\rmF_A\colon A\to A$ on commutative rings of characteristic~$p$, and apart from the well-known fact that Frobenius freely generates the Galois group of the prime field~$\bbF_p$, it has many other applications in algebra, arithmetic and even geometry. Even beyond fields, Frobenius is natural in all rings~$A$: Every map~$g\colon A\to B$ between commutative rings of characteristic~$p$ (i.e.~commutative~{\em$\bbF_p$-algebras}) commutes with Frobenius in the sense that the equation~\hbox{$g\circ\rmF_A=\rmF_B\circ g$} holds. In categorical terms, the Frobenius lies in the {\em center} of the category of commutative~$\bbF_p$-algebras. Furthermore, Frobenius freely generates the center: If~\hbox{$(\rmP_A\colon A\to A\,|\,A)$} is a family of ring maps such that
\begin{equation}\label{eq:commute}
g\circ\rmP_A=\rmP_B\circ g\tag{$\star$}
\end{equation}
holds for all~$g$ as above, then there exists an integer~$n\geqslant 0$ such that the equation~$\rmP_A=(\rmF_A)^n$ holds for all~$A$.

There are good reasons to consider a substantially richer category than the category of commutative~$\bbF_p$-algebras: the category of simplicial commutative~$\bbF_p$-algebras~\cite{Andre, Quillen:commutative_rings, Illusie:I, Illusie:II}. For example, the cotangent complexes in deformation theory are calculated using simplicial resolutions. In this more general context, one may wonder again what the natural operations are: families~\hbox{$(\rmP_A\colon A\to A\,|\,A)$} of simplicial ring maps as above such that~\eqref{eq:commute} holds for all simplicial ring maps~$g$. While the answer to this question is contained here as a special case of a part of our main result, there is more to the situation that should not be ignored. 

First, the simplicial structure on the objects leads to a simplicial structure on the set of all such operations, so that there even is a {\em space} of natural operations. Secondly, and more significantly, the simplicial structure  leads to an ambient homotopy theory, and the question one should better ask in this situation replaces the condition imposed by equation~\eqref{eq:commute} by a homotopy coherent structure. Homotopy coherent analogs of strict notions are obtained by replacing points in spaces that satisfy equalities such as~\eqref{eq:commute} by points, together with coherence data such as paths between both sides of the equations, and that is usually completed by even higher dimensional structure. For the notion of the center of a category, a coherent analog has been described in detail in~\cite{Szymik}, and it will also be reviewed here in Section~\ref{sec:centers}.


Our results apply to much more general situations than commutative~$\bbF_p$-algebras: We work in the context of algebraic theories in the sense of Lawvere~\cite{Lawvere}. This concept encompasses algebraic structures that are defined in terms of objects~$A$ together with operations~$A^n\to A$ for various~$n\geqslant0$, that satisfy certain equations. For example, groups, rings, Lie algebras, and many other standard algebraic structures can be encoded in this form. Because of their functorial nature, algebraic theories lend themselves well to extensions into homotopy theory. It seems that this has first been done in Reedy's~1974 thesis on the homology of algebraic theories. Later contributions are~\cite{Schwede} and~\cite{Badzioch}.

Our main result is Theorem~\ref{thm:centers_of_sAlg}. It says that for every (discrete) algebraic theory~$\Theta$, the homotopy coherent center of the category of simplicial~$\Theta$-algebras is homotopically discrete. In fact, the inclusion of the center of the category of discrete~$\Theta$-algebras is an equivalence. 

Our proof of the main theorem involves the calculation of homotopy coherent centers of categories of simplicial presheaves (Theorem~\ref{thm:centers_of_sPre}) as well as of (left Bousfield) localizations of simplicial categories in general (Theorem~\ref{thm:Bousfield_general}) and of some localizations of simplicial presheaves in particular (Theorem~\ref{thm:when_localization_preserves_centers}). These results have other applications as well and should therefore be of independent interest. In the related context of (presentable) quasi-categories, analogous results can be extracted from Lurie's work~\cite{Lurie}, as demonstrated in~\cite{Barwick+Schommer-Pries}. 
 

Here is an outline of the further contents. The first Section~\ref{sec:foundations} reviews some foundations of categories and higher categories as they are used here. In Section~\ref{sec:centers}, we recall the definition of the homotopy coherent center from~\cite{Szymik}. This is embedded in a discussion of spaces of homotopy coherent natural transformations, and Section~\ref{sec:classes_of_functors} provides some basic tools to manipulate these for various classes of functors. In Section~\ref{sec:presheaves} we discuss homotopy coherent centers for categories of simplicial presheaves, and in Section~\ref{sec:localizations} we do the same for left Bousfield localizations. On both of these pillars rests the final Section~\ref{sec:theories}, where we deal with our main subject of interest, the homotopy coherent centers of categories of (simplicial or homotopy) algebras for algebraic theories. Other applications, in addition to the ones already mentioned, will be spelled out in Sections~\ref{ssec:apps_for_sPre} and~\ref{ssec:apps_for_loc}.


\section{Homotopy theories}\label{sec:foundations}

The object of this paper is to calculate the homotopy coherent centers of certain large simplicial categories. In this section, we introduce the class of simplicial categories that we are interested in. These are the ones that come from combinatorial simplicial model categories, or what amounts to the same, from presentable quasi-categories. Both of these class are known to consist of localizations of categories of simplicial presheaves. We also explain our convention on largeness and other size issues in these contexts.

\subsection{Categories}

When speaking about categories, it is always possible to ignore the set-theoretic foundations. But, some of our results have implications for the sizes of the sets, categories, and spaces that we consider, and we would therefore rather have a means to address this. That can be done without undue effort by means of universes in the sense of~\cite{SGA4I}. 

The assumption that every set is contained in a universe leads to the existence of three non-empty universes~$\calU\in\calV\in\calW$. Given a choice of these, sets in~$\calU$ will be called {\em small}. Sets of the same cardinality as sets in~$\calU$ will be called {\em essentially small}. Sets in~$\calV$ will be called {\em large}.

A category is {\em (locally) small} if this holds for its set of morphisms (between any two objects), and similarly for large.  A large category is presentable if and and only if it is a suitable localization of a category of presheaves on a small category. When working with large categories, this will have to take place in the gigantic~$\calW$.


\subsection{Spaces}

Let~$\bfS=\Set_\Delta$ be the category of simplicial sets (or {\em spaces}). We are working with the usual notion of Kan equivalence between spaces. All objects are cofibrant, and the fibrant objects are the Kan complexes. They form the full subcategory~\hbox{$\bfS^\fib=\Set_\Delta^\fib$}. We write~$\Map(X,Y)$ for the space of maps between two spaces~$X$ and~$Y$, and we take care to ensure that~$Y$ is fibrant so that this is homotopically meaningful.

Some of the simplicial sets~$X$ that we construct are large. However, for all these~$X$ there will be a small simplicial set~$X'$ equivalent to~$X$, and the homotopy type of~$X'$ will not depend on the size of the universe in which the construction of the large~$X$ is carried out.


\subsection{Simplicial categories}

A {\em simplicial category}, for us, is a category~$\bfC$ that is enriched in spaces, and we will write~$\bfC(x,y)$ for the space of maps from the object~$x$ to the object~$y$. 

A simplicial category is {\em locally Kan} (or {\em fibrant}) if the mapping spaces~$\bfC(x,y)$ are Kan complexes for all choices of objects~$x$ and~$y$. The category~$\bfS$ with the usual mapping spaces does not have this property, only if we restrict our attention to Kan complexes. This problem arises for all simplicial model categories, where we have to ensure cofibrancy as well.

On the other hand, it will sometimes be convenient to assume (as in~\cite[II.2]{Goerss+Jardine}) that~$\bfC$ is bicomplete and that both of the functors~$\bfC(x,?)$ and~$\bfC(?,y)$ to spaces have left adjoints, so that~$\bfC$ is also tensored and cotensored with respect to~$\bfS$. These adjoints are usually written~$x\otimes?$ and~$y^?$, so that there are isomorphisms
\[
\bfS(K,\bfC(x,y))\cong\bfC(x\otimes K,y)\cong\bfC(x,y^K)
\]
of spaces for all spaces~$K$. This is the case for simplicial model categories.

Typically, we will be working in a simplicial category~$\bfC$ of the form~$\bfM^\fc$ for some nice simplicial model category~$\bfM$. The details are as follows.


\subsection{Combinatorial simplicial model categories}

All of the large simplicial categories that we will be dealing with in this paper have additional structure: they are embedded into simplicial model categories in the sense of Quillen~\cite{Quillen:Homotopical}. Recall that a Quillen model category~$\bfM$ is {\em simplicial} if it is also a simplicial category so that axiom SM7 holds: For any cofibration~$j\colon a\to b$ and any fibration~\hbox{$q\colon x\to y$}, the canonical map
\[
\bfM(b,x)\longrightarrow\bfM(a,x)\hspace{-1em}\underset{\displaystyle\bfM(a,y)}{\times}\hspace{-1em}\bfM(b,y)
\]
is a fibration, which is trivial if~$j$ or~$q$ is. If~$\bfM$ is a simplicial Quillen model category, then the full subcategory~$\bfM^\fc$ of fibrant and cofibrant objects is a locally Kan simplicial category.


For practical purposes, the class of simplicial Quillen model categories is still too big. The subclass of combinatorial model categories has been singled out by Jeff Smith. We follow the presentation in~\cite{Dugger:presentations}: A model category~$\bfM$ is called {\em combinatorial} if it is cofibrantly-generated and the underlying category is (locally) presentable. Dugger has shown that, up to Quillen equivalence, we can assume that these are simplicial as well. 

Combinatorial model categories admit fibrant/cofibrant replacement functors, and we will use these in order push certain constructions which{\it~a priori} land in~$\bfM$ into~$\bfM^\fc$.  

\begin{examples}\label{ex:all_csmc}
If~$\bfC$ is a small category, then the category~$\sPre(\bfC)$ of simplicial presheaves with pointwise equivalences and fibrations is a combinatorial simplicial model category. For every~(small) set~$W$ of maps~$\sPre(\bfC)$, one can form the left Bousfield localization~$\sPre(\bfC)[W^{-1}]$ in which the maps from~$W$ have been added to the equivalences, and this is also a combinatorial simplicial model category. For more information we refer to the later Section~\ref{sec:presheaves} and Section~\ref{sec:localizations}, where we study categories of presheaves and their localizations in more detail.
\end{examples}

Dugger's theorem \cite[1.1,~1.2]{Dugger:presentations} states that up to Quillen equivalence every combinatorial~(simplicial) model category~$\bfM$ is of the form~$\Pre(\bfC)[W^{-1}]$ for suitable~$\bfC$ and~$W$. One can even achieve that all objects are cofibrant, and the model structure is left proper: co-base changes of equivalences along cofibrations are equivalences. 


\section{Homotopy coherent centers}\label{sec:centers}

In this section we set up some notation and recall the definition of the homotopy coherent center of a simplicial category from~\cite{Szymik}.


\subsection{Strict natural transformations}

Let~$\bfC$ and~$\bfD$ be simplicial categories. A {\em simplicial functor}~$F\colon\bfC\to\bfD$ is given by a map~$F$ on objects together with simplicial maps~$\bfC(x,y)\to\bfD(Fx,Fy)$ that preserve the identities and the composition. The identity functor on~$\bfC$ will be denoted by~$\Id_\bfC$.

The {\em simplicial natural transformations} between two functors~$F,G\colon\bfC\to\bfD$ form a space~$\Nat(F,G)$, the equalizer of an evident pair of simplicial maps
\[
\prod_x\bfD(Fx,Gx)\Longrightarrow\prod_{y,z}\Map(\bfC(y,z),\bfD(Fy,Gz)),
\]
where the indices~$x$,~$y$,~$z$ run through the objects in~$\bfC$. Equivalently, the space~$\Nat(F,G)$ is the limit of a (similarly evident) cosimplicial space (that we recall in Section~\ref{subsec:N1} below and) that extends the parallel pair displayed above.

The center of an ordinary category is the monoid of all natural transformations from the identity functor to itself. In the simplicial context, we can imitate this definition, and it leads to the notion of the {\em simplicial center}
\[
\Z(\bfC)=\Nat(\Id_\bfC,\Id_\bfC)
\]
of~$\bfC$, see~\cite[Definition~1.1]{Szymik}. This is a simplicial abelian monoid.


The {\em homotopy category}~$\Ho\bfC$ of a simplicial category~$\bfC$ has the same objects but the mapping spaces are replaced by their sets of components. A morphism~\hbox{$f\colon x\to y$} in a simplicial category~$\bfC$ is called an {\em equivalence} if and only if it represents an isomorphism in the homotopy category. Two objects are {\it equivalent} if they are isomorphic in the homotopy category, i.e.~if and only if there exists a zigzag of equivalences between them.

An {\it equivalence} of simplicial categories is a simplicial functor that is homotopically full and faithful (in the sense of Definition~\ref{def:homotopically_ff} below) and induces an equivalence of homotopy categories. Two simplicial categories are {\it equivalent} if and only there exists a zigzag of equivalences between them. 

As usual in homotopy theory, we are only interested in things up to equivalence, and the following two definitions make this precise in the context of simplicial functors: Two simplicial functors~$F,G\colon\bfC\to\bfD$ are {\em equivalent} if there exists a zigzag of simplicial natural transformation between them which are objectwise equivalences~(in~$\bfD$). 

This reproduces the usual notion of equivalence for simplicial presheaves (see Section~\ref{sec:presheaves}), so that the Yoneda embedding preserves and detects, i.e.~reflects, equivalences. Equivalences of functors will be denoted by the omnipresent `$\simeq$' symbol. As an application of this terminology, we can make the following 

\begin{definition}\label{def:homotopy_adjunction}
A {\em homotopy adjunction} between simplicial categories~$\bfC$ and~$\bfD$ is a pair of simplicial functors~$L\colon\bfC\leftrightarrow\bfD\colon R$ together with an equivalence
\[
\bfD(Lc,d)\simeq\bfC(c,Rd)
\]
of functors~$\bfC^\op\times\bfD\to\bfS$.
\end{definition}


\subsection{Homotopy coherent natural transformations}\label{subsec:N1}

We will use classical explicit models for spaces of homotopy coherent natural transformations between simplicial functors. See~\cite{Cordier+Porter:Equivariant},~\cite{Cordier+Porter:TAMS} for this and generalizations. Compare also~\cite[Section~2]{Arone+Dwyer+Lesh}, which contains a discussion of spaces of natural transformations from the derived functor perspective in the case when the target category is either the category of spaces of that of spectra.

Let~$F,G\colon\bfC\to\bf\bfD$ be simplicial functors between simplicial categories, and assume that the target~$\bfD$ is locally Kan. For any integer~$n\geqslant0$ we can consider the space
\[
\Pi^n(F,G)=\prod_{x_0,\dots,x_n}
\Map( \bfC(x_1,x_0)\times\dots\times\bfC(x_n,x_{n-1}) , \bfD(Fx_n,Gx_0) )
\]
where the product runs over the~$(n+1)$-tuples of objects of~$\bfC$. Together with the evident structure maps, this defines a cosimplicial space~$\Pi^\bullet(F,G)$. The {\em space of homotopy coherent natural transformations}
\begin{equation}\label{eq:explicit_model}
\hoNat(F,G)=\Tot\Pi^\bullet(F,G)
\end{equation}
is defined as the totalization of the cosimplicial space~$\Pi^\bullet(F,G)$.

If the target $\bfD$ is not necessarily locally Kan, then we can choose a fibrant replacement~\hbox{$r\colon\bfD\to\bfD'$} of it, and we can set
\[
\hoNat(F,G)=\hoNat(rF,rG)=\Tot\Pi^\bullet(rF,rG).
\]
This does not depend on the choice of $r$ up to a contractible choice of equivalence.

The following is immediate from the fact that a level-wise equivalence between fibrant cosimplicial spaces (in the sense of Bousfield and Kan) induces an equivalence between their totalizations.

\begin{proposition}
All natural equivalences~$F\to F'$ and~$G\to G'$ between simplicial functors~\hbox{$F,F',G,G'\colon\bfC\to\bfD$} induce equivalences~$\hoNat(F',G)\to\hoNat(F,G)$ and~\hbox{$\hoNat(F,G)\to\hoNat(F,G')$} between the spaces of homotopy coherent natural transformations.
\end{proposition}


\subsection{Homotopy coherent centers}\label{subsec:N2}

Let $\bfC$ be locally Kan simplicial category. Specializing the consideration of the previous Section~\ref{subsec:N1} to the case~\hbox{$F=\Id_\bfC=G$}, this defines the {\em homotopy coherent center} of~$\bfC$ as
\begin{equation}\label{def:Z}
\hoZ(\bfC)=\hoNat(\Id_\bfC,\Id_\bfC),
\end{equation}
see~\cite[Definition~2.3]{Szymik}. 

To extend this definition to general simplicial categories $\bfC$, we can choose a fibrant replacement~\hbox{$r\colon\bfC\to\bfC'$}. For example, keeping the same objects, the function complexes in $\bfC'$ can be taken to be the singular complexes on the geometric realizations of the function complexes in $\bfC$. The previous definitions result in
\[
\hoZ(\bfC)=\hoNat(\Id_\bfC,\Id_\bfC)=\hoNat(r,r),
\]
and by direct inspection of the definitions (compare Proposition~\ref{prop:homotopically_dense}) this is also equivalent to $\hoNat(\Id_{\bfC'},\Id_{\bfC'})=\hoZ(\bfC')$. 

Equivalent locally Kan simplicial categories have equivalent homotopy coherent centers, see~\cite[Theorem~4.1]{Szymik}.


\subsection{The case of Quillen model categories}\label{subsec:bar}

In the situation most relevant to us, the locally Kan simplicial category~$\bfD$ embeds as~$\bfM^\fc$ into some nice simplicial model category~$\bfM$. We can then give a more conceptual definition of the spaces of homotopy coherent natural transformation into~$\bfM^\fc$, using the additional structure present in the ambient category~$\bfM$.

Let~$F\colon\bfC\to\bf\bfM^\fc$ be a simplicial functor between simplicial categories, where~$\bfM$ is tensored over~$\bfS$, and let~\hbox{$W\colon\bfC^\op\times\bfC'\to\bfS$} be another simplicial functor. It plays the role of a `weight bimodule.' Typical examples will be~$W(x,y)=\bfC(x,y)$ for~\hbox{$\bfC'=\bfC$}, or more generally~$W(x,y)=\bfC'(Ex,y)$ for a functor~\hbox{$E\colon\bfC\to\bfC'$}. These will be denoted by~$\bfC_{\Id}$ and~$\bfC'_E$, respectively. 

Recall that the {\em bar construction}~$\rmB_\bullet(W,\bfC,F)$ is a simplicial object in the category of simplicial functors~$\bfC\to\bfM^\fc$ with
\[
\rmB_n(W,\bfC,F)=\sum_{x_0,\dots,x_n}\left(W(x_0,?)\times\bfC(x_1,x_0)\times\dots\times\bfC(x_n,x_{n-1})\right)\otimes F(x_n).
\]
Clearly, this also requires the existence of enough colimits. The tensor (or coend)~$W\otimes_\bfC F$ is the colimit of~$\rmB_\bullet(W,\bfC,F)$. It can also be written as the coequalizer of the first two face maps. In general, the colimit of a simplicial space~$X_\bullet$ is~$\star\otimes_\Delta X_\bullet$, whereas~$\Delta^\bullet\otimes_\Delta X_\bullet$ is its geometric realization. In particular, the geometric realization~$\rmB(W,\bfC,F)$ of the bar construction is~$\Delta^\bullet\otimes_\Delta\rmB_\bullet(W,\bfC,F)$. Useful references for the bar construction calculus in enriched contexts are~\cite{Hollender+Vogt} and~\cite{Meyer}.

Given a functor~$F$, the bar construction~$\rmB_\bullet(\bfC_{\Id},\bfC,F)$ is a kind of  simplicial resolution of it, and given another simplicial functor~$G\colon\bfC\to\bfM^\fc$, this gives rise to the cosimplicial space~$\Nat(\rmB_\bullet(\bfC_{\Id},\bfC,F),G)$ with level~$n$ isomorphic with
\begin{equation}\label{eq:without}
\prod_{x_0,\dots,x_n}\Map(\bfC(x_1,x_0)\times\dots\times\bfC(x_n,x_{n-1}),\bfM(F(x_n),G(x_0)))
\end{equation}
by the Yoneda lemma. Its totalization
\begin{equation}
\hoNat(F,G)=\Tot(\Nat(\rmB_\bullet(\bfC_{\Id},\bfC,F),G))
\end{equation}
is the explicit model~\eqref{eq:explicit_model}.

It is clear from~\eqref{eq:explicit_model} that~$\hoNat(F,G)$ is already meaningful if~$F$ and~$G$ take values in cofibrant and fibrant objects, respectively. This flexibility will come in handy later.

If~$\bfM$ is a simplicial model category, one might be tempted to consider it as a simplicial category, thereby forgetting about the Quillen model structure. Then the space
\[
\hoZ(\bfM)
\]
can be computed as the homotopy coherent center of an equivalent locally Kan simplicial category. However, this is not the right thing to do. Looking at~$\bfM$~(even up to equivalence of simplicial categories) involves ignoring the homotopy theory~(the weak equivalences) of~$\bfM$. To take this structure into account, we will have to pass to a (locally Kan) simplicial category that codifies the homotopy theory of~$\bfM$. A canonical choice is the simplicial category~$\bfM^\fc$ of fibrant/cofibrant objects in~$\bfM$. Another one would be (a locally Kan replacement of) the simplicial localization~$\rmL(\bfM)$ of~$\bfM$, see~\cite{Dwyer+Kan:Simplicial}, \cite{Dwyer+Kan:Calculating}, and~\cite{Dwyer+Kan:Function}. Summing up, this justifies to consider
\[
\hoZ(\bfM^\fc)
\]
as a suitable model for the homotopy coherent center of the homotopy theory defined by a simplicial Quillen model category~$\bfM$.

Note that~$\hoZ(\bfM^\fc)\simeq\hoZ(\bfN^\fc)$ if~$\bfM$ and~$\bfN$ are Quillen equivalent, since~$\bfM^\fc$ and~$\bfN^\fc$ are equivalent as simplicial categories in this case.


\subsection{A digression on \texorpdfstring{$\infty$}{infinity}-categories}

Besides simplicial categories, another popular model for $\infty$-categories are quasi-categories. These were known to Boardman and Vogt~\cite{Boardman+Vogt} under the name {\it weak Kan complexes}. They have been developed as a model for~$\infty$-categories by Joyal~\cite{Joyal} and Lurie~\cite{Lurie}. It is a very economic model in the sense that the category of quasi-categories sits inside the category of simplicial sets as the fibrant and cofibrant objects for Joyal's model category structure; one might say that quasi-categories {\em are} their nerves.

Cordier's simplicial (or coherent) nerve construction~$\rmN_\Delta$ for simplicial categories~\cite{Cordier} extends the usual nerve~$\rmN$ for ordinary categories, and it is part of a Quillen equivalence between simplicial categories (with the model structure of Bergner~\cite{Bergner}) and Joyal's model category~\cite[2.2.5]{Lurie}. In particular, if~$\bfC$ is locally Kan, then~$\rmN_\Delta(\bfC)$ is a quasi-category. This holds, in particular, if~$\bfC=\bfM^\fc$ for some simplicial model category~$\bfM$. While this suggests, that quasi-categories are more general than simplicial categories, this is not the case. Lurie has shown that presentable quasi-categories are, up to equivalence, all of the form~$\rmN_\Delta(\bfM^\fc)$ for combinatorial simplicial model categories~$\bfM$. His proof uses Dugger's theorem, see~\cite[A.3.7.6]{Lurie}. In other words, the homotopy theories most commonly considered (both in the context of combinatorial model categories and presentable quasi-categories) are all equivalent to localizations of presheaf categories. This justifies explains our {\it modus procedendi} for the computation of homotopy coherent centers.


In the context of a quasi-category~$\bfQ$, the homotopy coherent center of~$\bfQ$ is the endomorphism space of the identity functor in the quasi-category of endofunctors~\hbox{$\bfQ\to\bfQ$}. This agrees with our definition for locally Kan simplicial categories under the cited Quillen equivalence.


\section{Some distinguished classes of functors}\label{sec:classes_of_functors}

In this section we study the behavior of spaces of homotopy coherent natural transformations with respect to three distinguished classes of functors: Homotopically full and faithful functors, homotopy left Kan extensions, and homotopically dense functors. The ensuing calculus for spaces of spaces of homotopy coherent natural transformations will later enable us to determine the homotopy coherent centers for the categories that we are interested in.

\subsection{Homotopically full and faithful functors}

Recall that a simplicial functor~$F\colon\bfD\to\bfT$ is called {\em full and faithful}, if for all~$x$ and~$y$ the induced maps~$\bfD(x,y)\to\bfT(Fx,Fy)$ are isomorphisms. Given simplicial functors~\hbox{$F\colon\bfD\to\bfT$} and~$G,H\colon\bfC\to\bfD$, the induced map
\[
\Nat(G,H)\longrightarrow\Nat(FG,FH)
\]
is an isomorphism whenever~$F$ is full and faithful. Here is the analog of that notion that is appropriate in homotopical contexts.

\begin{definition}\label{def:homotopically_ff}
A simplicial functor~$F\colon\bfD\to\bfT$ is called {\em homotopically full and faithful}, if the induced maps~$\bfD(x,y)\to\bfT(Fx,Fy)$ are equivalences for all~$x$ and~$y$. 
\end{definition}

In the homotopical situation this is not obviously the compositum of a~`full' and a~`faithful' part. 

\begin{proposition}\label{prop:ff_implies_hff}
Full and faithful functors are also homotopically full and faithful.
\end{proposition}

We record the following easy observation that will be helpful in our calculations.

\begin{proposition}{\upshape\bf(Cancellation of homotopically full and faithful functors)}\label{prop:cancellation}
Let~$\bfC$,~$\bfD$, and~$\bfT$ be simplicial categories with~$\bfD$ and~$\bfT$ locally Kan.
Given simplicial functors~\hbox{$F\colon\bfD\to\bfT$} and~$G,H\colon\bfC\to\bfD$ such that~$F$ is homotopically full and faithful, the induced map
\[
\hoNat(G,H)\longrightarrow\hoNat(FG,FH)
\]
of spaces is an equivalence.
\end{proposition}

\begin{proof}
By hypothesis, the simplicial functor~$F$ induces natural equivalences
\[
\bfD(Gx,Hy)\overset{\simeq}{\longrightarrow}\bfT(FGx,FHy)
\]
for all objects~$x$ and~$y$. These induce equivalences
\[
\Pi^n(G,H)\overset{\simeq}{\longrightarrow}\Pi^n(FG,FH)
\]
between the levels of the cosimplicial spaces, and therefore an equivalence between the totalizations of these. 
\end{proof}

Taking~$G=\Id=H$ gives the following means for calculating homotopy coherent centers.

\begin{corollary}\label{cor:full_and_faithful_and_centers}
For all homotopically full and faithful functors~$F\colon\bfC\to\bfD$ between simplicial categories that are locally Kan, there exists an equivalence
\[
\hoZ(\bfC)\simeq\hoNat(F,F).
\]
\end{corollary}


\subsection{Homotopy left Kan extensions}

We consider a diagram
\begin{equation}\label{eq:Kan_diagram}
\xymatrix{
\bfC\ar[d]_-F\ar[r]^-Q&\bfT\\
\bfD
}
\end{equation}
of simplicial functors. Recall, for example from~\cite[Chapter~4]{Kelly}, that a {\em simplicial left Kan extension}~$\LKan_FQ$ is a simplicial functor~\hbox{$\bfD\to\bfT$} together with a natural isomorphism
\[
\Nat(\LKan_FQ,R)\cong\Nat(Q,F^*R)
\]
of spaces of simplicial natural transformations. This can be rephrased to say that the functor~$\LKan_F\colon\bfT^\bfC\to\bfT^\bfD$ between functor categories is left adjoint (in the simplicial sense) to the restriction functor~$F^*\colon\bfT^\bfD\to\bfT^\bfC$ with~$F^*R=RF$. 

We will now pass to the homotopy coherent version of the left Kan extension, and we will therefore assume that~$\bfT=\bfM^\fc$ for some simplicial model category~$\bfM$. Then homotopy colimits and more generally realizations of simplicial objects in~$\bfT$ are available, since these can be constructed in~$\bfM$ and then transferred to~$\bfT$ by fibrant/cofibrant replacement.

Using the functor~$F$ to define the weight bimodule~$\bfD_F$ with~$\bfD_F(c)=\bfD(Fc,?)$ as in Section~\ref{subsec:bar}, the simplicial left Kan extension can be defined (compare~\cite[Section~4.1]{Kelly}) as a tensor~(or coend)
\[
(\LKan_FQ)(d)=\bfD_F\otimes_\bfC Q.
\]
Consequently, it should not be a surprise that homotopy left Kan extensions can be defined by using the corresponding homotopy coend, which is the bar construction:
\[
(\hoLKan_FQ)(d)=\rmB(\bfD_F, \bfC, Q).
\]
See~\cite{Cordier+Porter:TAMS}, for example. Similar to Proposition~6.1 of~{\it loc.cit.}, we have

\begin{proposition}{\upshape\bf(Universal property of homotopy left Kan extensions)}\label{prop:Kan_extensions}
In the situation~\eqref{eq:Kan_diagram}, there exists an equivalence
\[
\hoNat(\hoLKan_FQ,R)\simeq\hoNat(Q,F^*R)
\]
of spaces of homotopy coherent natural transformations.
\end{proposition}

Before we give a proof, let us mention that, for the spaces a homotopy coherent natural transformations to be meaningful, we should assume that~$R\colon\bfC\to\bfM$ lands into~$\bfM^\fc$. Since we already assume that~$Q$ does so, it follows that~$\hoLKan_FQ$ takes values in cofibrant objects, and this is enough to serve our purposes.

\begin{proof}
On the one hand, we have
\begin{align*}
\hoNat(\hoLKan_FQ,R)
&\cong\Tot\Nat(\rmB_\bullet(\bfD_{\Id},\bfD,\hoLKan_FQ),R)\\
&\cong\Tot\Nat(\rmB_\bullet(\bfD_{\Id},\bfD,\rmB(\bfD_F,\bfC,Q)),R)\\
&\cong\Tot\Nat(\rmB_\bullet(\rmB(\bfD_{\Id},\bfD,\bfD_F),\bfC,Q),R),
\end{align*}
while on the other hand, we have
\begin{align*}
\hoNat(Q,F^*R) 
&\cong\Tot\Nat(\rmB_\bullet(\bfC_{\Id},\bfC,Q),F^*R)\\
&\cong\Tot\Nat(\LKan_F\rmB_\bullet(\bfC_{\Id},\bfC,Q),R)\\
&\cong\Tot\Nat(\bfD_F\otimes_\bfC\rmB_\bullet(\bfC_{\Id},\bfC,Q),R)\\
&\cong\Tot\Nat(\rmB_\bullet(\bfD_F,\bfC,Q),R),
\end{align*}
and there exists an equivalence between the two of them induced by the equivalence
\[
\rmB(\bfD_{\Id},\bfD,\bfD_F)\overset{\sim}{\longrightarrow}\bfD_F
\]
of functors.
\end{proof}


\subsection{Homotopically dense functors}\label{ssec:dense}

Recall, for example from \cite[5.1]{Kelly}, that a simplicial functor~$F\colon\bfC\to\bfD$ is called {\em dense} if there exists an isomorphism
\[
\Id_\bfD\cong\LKan_FF
\]
as functors~$\bfD\to\bfD$. There exists a useful characterization of dense functors~$F\colon\bfC\to\bfD$ as the functors such that the associated functor
\begin{equation}\label{eq:tilde}
\widetilde F\colon\bfD\longrightarrow\sPre(\bfC),\,d\longmapsto\bfD(F?,d)
\end{equation}
is both full and faithful, see~\cite[Theorem~5.1]{Kelly}. Here, the target category is the category~$\sPre(\bfC)$ of simplicial presheaves on~$\bfC$. (See Section~\ref{sec:presheaves} for more on presheaves.) This associated functor is the composition~$\sPre(F)\circ\rmY_\bfD$, where the functor~\hbox{$\sPre(F)\colon\sPre(\bfD)\to\sPre(\bfC)$} is induced by the functor~$F$, and~$\rmY_\bfD$ is the Yoneda embedding of~$\bfD$.  It can also be described as adjoint to the weight bimodule~$\bfD_F$.

In the homotopy coherent setting, we proceed as follows. Let~$F\colon\bfC\to\bfM^\fc$
be a simplicial functor, and denote by~$F'$ the composition of~$F$ with the inclusion~\hbox{$\bfM^\fc\to\bfM$}. We can then compute~$\hoLKan_FF'\colon\bfM^\fc\to\bfM$ as above by realizing a bar construction, and then~$\hoLKan_FF$ is the composition of a fibrant/cofibrant replacement functor~$\bfM\to\bfM^\fc$ with~$\hoLKan_FF'$. 

\begin{definition}\label{def:homotopical_dense}
A simplicial functor~$F\colon\bfC\to\bfM^\fc$ is {\em homotopically dense} if there exists an equivalence
\[
\Id_{\bfM^\fc}\simeq\hoLKan_FF
\]
as functors~$\bfM^\fc\to\bfM^\fc$.
\end{definition}

With the definition in place, we can now move on to explain the consequences of homotopical density for the purpose of calculations of spaces of homotopy coherent natural transformations and homotopy coherent centers.

\begin{proposition}\label{prop:homotopically_dense}
If the functor~$F\colon\bfC\to\bfM^\fc$ is homotopically dense, then there exists equivalences
\[
\hoNat(\Id_{\bfM^\fc}, G)\simeq\hoNat(F,GF)
\]
for all functors~\hbox{$G\colon\bfM^\fc\to\bfM^\fc$}.
\end{proposition}

\begin{proof}
By hypothesis, the identity functor on~$\bfM^\fc$ is a homotopy left Kan extension of~$F$ along itself. This gives
\[
\hoNat(\Id_{\bfM^\fc}, G)\simeq\hoNat(\hoLKan_FF,G).
\]
By the adjunction property of the left Kan extension, we have
\[
\hoNat(\hoLKan_FF,G)\simeq\hoNat(F,GF).
\]
Together, these prove the result.
\end{proof}

Taking~$G=\Id_{\bfM^\fc}$ as well, the preceding result has the following consequence.

\begin{corollary}\label{cor:dense_and_centers}
For all homotopically dense functors~$F\colon\bfC\to\bfM^\fc$, there exists an equivalence
\[
\hoZ(\bfM^\fc)\simeq\hoNat(F,F).
\]
\end{corollary}

Putting this together with Corollary~\ref{cor:full_and_faithful_and_centers}, we obtain the next   result.

\begin{proposition}\label{prop:twin_criterion}
For all functors~$F\colon\bfC\to\bfM^\fc$ that are both homotopically full and faithful as well as homotopically dense, the homotopy coherent centers of~$\bfC$ and~$\bfM^\fc$ are equivalent.
\end{proposition}

The next section will feature an important class of functors for which this result applies, the Yoneda embeddings.


\section{Presheaves}\label{sec:presheaves}

This section discusses homotopy coherent centers for diagram categories. There exists a choice as to the variance of the diagrams (covariant and contravariant), and we have chosen to use contravariant functors, also known as presheaves. We begin by recalling them in the simplicial setting.


\subsection{Simplicial presheaves}

Let~$\bfC$ be a simplicial category in the sense of Section~\ref{sec:centers}. Then a {\em presheaf} on~$\bfC$ is a simplicial functor~\hbox{$M\colon\bfC^\op\to\bfS$} into the simplicial category~$\bfS=\sSet$ of spaces/simplicial sets. The presheaves together with their simplicial natural transformations form a simplicial category~$\sPre(\bfC)$.

If the category~$\bfC$ is discrete, then a simplicial presheaf can be thought of both as a presheaf of spaces, or as a simplicial object in the category of sheaves (of sets). It follows that in this case~$\sPre(\bfC)$ is indeed the category of simplicial objects in the category~$\Pre(\bfC)$ of presheaves on~$\bfC$.

\begin{remark}\label{rem:sPre_model}
The category~$\sPre(\bfC)$ of simplicial presheaves comes with simplicial Quillen model structures such that the equivalences are formed objectwise. In one such, the {\it projective} one, the fibrations are also formed objectwise, and the fibrant sheaves are those that take values in Kan complexes. See~\cite[Proposition~IX.1.4]{Goerss+Jardine}, for example. 
\end{remark}


\subsection{Derived mapping spaces for simplicial presheaves}

If~$F$,~$G$ are presheaves of Kan complexes on a simplicial category~$\bfC$, then there are at least two ways of thinking about the derived space of maps~$F\to G$: 

On the one hand, we can think of~$F$ and~$G$ as simplicial functors between simplicial categories, and then we have defined a space~$\hoNat(F,G)$ of homotopy coherent natural transformations between them as in~\eqref{eq:explicit_model}.

On the other hand, using the projective Quillen model structure on the category~$\bfP=\sPre(\bfC)$ of presheaves on~$\bfC$, where the equivalences and fibrations are defined object-wise, the presheaves of Kan complexes are precisely the fibrant objects. Therefore, the derived mapping space should be~$\bfP(F',G)$, where~$F'$ is a cofibrant replacement of~$F$.

The following result assures that the resulting spaces are equivalent.

\begin{proposition}
If~$F,G$ are two presheaves of Kan complexes on a simplicial category~$\bfC$, then there the space~$\hoNat(F,G)$ is equivalent to the derived space of morphisms~$F\to G$ in the projective Quillen model structure on~$\sPre(\bfC)$.
\end{proposition}

\begin{proof}
The definition of~$\hoNat(F,G)$ is~$\Nat(\rmB_\bullet(\bfC_{\Id},\bfC,F),G)$, where~$\rmB_\bullet(\bfC_{\Id},\bfC,F)$ is the bar construction on~$F$. See Section~\ref{subsec:bar}. Now~$\Nat(\rmB_\bullet(\bfC_{\Id},\bfC,F),G)$ is the actual mapping space in~$\sPre(\bfC)$. We have to show that~$\rmB_\bullet(\bfC_{\Id},\bfC,F)$ is cofibrant and that~$G$ is fibrant in the Quillen model structure under consideration. This is clear for the latter, since~$G$ is object-wise fibrant by assumption. It remains to be seen that~$\rmB_\bullet(\bfC_{\Id},\bfC,F)$ is a Reedy cofibrant simplicial diagram, i.e.~that the latching maps are cofibrations. But these are coproducts of inclusions of the degenerate parts (where some factor is an identity) of products of mapping spaces. These maps are cofibrations because the functor~$x\mapsto\bfC_{\Id}(x,y)$ is cofibrant when thought of as a (representable) presheaf.
\end{proof}


\subsection{The Yoneda embedding}\label{ssec:Yoneda}

If~$\bfC$ is a simplicial category, then the Yoneda embedding
\[
\rmY=\rmY_\bfC\colon\bfC\longrightarrow\sPre(\bfC)
\]
sends each object~$x$ to the representable functor
\[
\rmY_\bfC(x)\colon t\longmapsto\bfC(t,x)
\]
that it defines. Therefore, if~$\bfC$ is locally Kan, then the values of the presheaf~$\rmY_\bfC(x)$ are Kan complexes, so that for every object~$x$ of~$\bfC$ the presheaf~$\rmY_\bfC(x)$ is fibrant in the usual Quillen model category on~$\sPre(\bfC)$.

The enriched Yoneda lemma asserts that the space of maps~$\rmY_\bfC(x)\to M$ is isomorphic to the space~$M(x)$ via the evaluation at the identity. For example, it follows that the image of the Yoneda embedding consists of cofibrant presheaves in the Quillen model structure of Remark~\ref{rem:sPre_model}.

The Yoneda lemma also implies that the Yoneda embedding deserves its name: it is full and faithful. 

The Yoneda embedding will not be essentially surjective in general, but it is always dense in the sense of Section~\ref{ssec:dense}:
\[
\Id_{\sPre(\bfC)}\cong\LKan_\rmY\rmY
\]
as functors~$\sPre(\bfC)\to\sPre(\bfC)$. This is perhaps most easily seen by the characterization of dense functors~$F\colon\bfC\to\bfD$ by the full and faithfulness of the functor
\[
\bfD\longrightarrow\sPre(\bfC),\,d\longmapsto\bfD(F?,d).
\]
If~$F=\rmY_\bfC$ is the Yoneda embedding, this functor is isomorphic to the identity, again by the Yoneda lemma.

Full and faithful functors are also homotopically so by Proposition~\ref{prop:ff_implies_hff}. In particular, the Yoneda embedding is homotopically full and faithful. 

There is no obvious reason why the corresponding result for dense functors should also be true, so that the following requires proof.

\begin{proposition}\label{prop:Yoneda_homotopically_dense}
Every Yoneda embedding is homotopically dense.
\end{proposition}

\begin{proof}
Let~$\bfC$ be a simplicial category. To enhance readability in the following calculation, let us again abbreviate~$\bfP=\sPre(\bfC)$ for its category of simplicial presheaves, and~$\rmY=\rmY_\bfC$ for its Yoneda embedding~$\rmY\colon\bfC\to\bfP$. We have to show that~$\hoLKan_\rmY\rmY\simeq\Id_\bfP$. 

We start with the definition~$\hoLKan_\rmY\rmY=\rmB(\bfP_\rmY,\bfC,\rmY)$, where~$\bfP_\rmY$ is our notation for the bimodule~$\bfP_\rmY(x,f)=\bfP(\rmY x,f)$ for~$x\in\bfC$ and~$f\in\bfP$. We introduce its~`dual' bimodule, which is~$\bfP^\op_\rmY(f,x)=\bfP(f,\rmY x)$. 
Composition of the equivalence
\[
\rmB(\bfP_{\Id},\bfP,\Id_\bfP)\simeq\Id_\bfP
\]
with the Yoneda embedding gives the evaluation equivalence
\[
\rmB(\bfP^\op_\rmY,\bfP,\Id_\bfP)\simeq\rmY.
\]
Similarly, there is a composition equivalence
\[
\rmB(\bfP_\rmY,\bfC,\bfP^\op_\rmY)\simeq\bfP_{\Id}.
\]
Using these, we can calculate
\begin{align*}
\hoLKan_\rmY\rmY
&=\rmB(\bfP_\rmY,\bfC,\rmY)\\
&\cong\rmB(\bfP_\rmY,\bfC,\rmB(\bfP^\op_\rmY,\bfP,\Id_\bfP))\\
&\cong\rmB(\rmB(\bfP_\rmY,\bfC,\bfP^\op_\rmY),\bfP,\Id_\bfP)\\
&\cong\rmB(\bfP_{\Id},\bfP,\Id_\bfP),
\end{align*}
and this is indeed equivalent to~$\Id_\bfP$, as desired.
\end{proof}


\subsection{Centers of categories of presheaves}

We are now ready to prove the main result of this section.

\parbox{\linewidth}{\begin{theorem}{\upshape\bf(Morita invariance)}\label{thm:centers_of_sPre}
The Yoneda embedding defines an equivalence
\[
\hoZ(\sPre(\bfC)^\fc)\simeq\hoZ(\bfC)
\]
of homotopy coherent centers.
\end{theorem}}

\begin{remark}
According to our conventions, the simplicial category~$\bfC$ in the statement is assumed to be locally Kan. However, this is not really necessary: If~\hbox{$\bfC\to\bfC'$} is an equivalence of simplicial categories, where~$\bfC'$ is locally Kan, then~$\sPre(\bfC)$ is Quillen equivalent to~$\sPre(\bfC')$, and so these two presheaf Quillen model categories have the same homotopy coherent centers.
\end{remark}

\begin{proof}
Since the Yoneda embedding is both homotopically full and faithful as well as homotopically dense, Proposition~\ref{prop:twin_criterion} applies to give the result.
\end{proof}

We note that the proof has more precisely shown that both~$\hoZ(\sPre(\bfC)^\fc)$ and~$\hoZ(\bfC)$  agree with~$\hoNat(\rmY_\bfC,\rmY_\bfC)$.

As a consequence of Theorem~\ref{thm:centers_of_sPre}, we see that the homotopy coherent center of the large category~$\sPre(\bfC)$ is essentially small. 

The nickname of Theorem~\ref{thm:centers_of_sPre} stems from the following corollary.

\begin{corollary}
The homotopy coherent center of a simplicial category~$\bfC$ depends only on the category of simplicial presheaves on~$\bfC$. 
\end{corollary}


\subsection{Applications}\label{ssec:apps_for_sPre}

Theorem~\ref{thm:centers_of_sPre} has numerous applications. We list some of the immediate ones.

\begin{example}\label{ex:spaces}
The category~$\bfS=\sSet$ of spaces/simplicial sets can be written as the category~$\sPre(\star)$ of sheaves on the singleton~$\star$. While it is an easy exercise to see that the simplicial center~$\Z(\bfS)$ is trivial, Theorem~\ref{thm:centers_of_sPre} immediately implies that also the homotopy coherent center~$\hoZ(\bfS^\fc)$ of the simplicial category~$\bfS^\fc$ of Kan complexes is contractible. While this result may not be surprising, it is certainly not clear from the definition of the homotopy coherent center.
\end{example}

For the category of {\it pointed} spaces, see Example~\ref{ex:pointed}.

\begin{example}
More generally, if~$\bfG$ is a simplicial group, then~$\sPre(\bfG)$ is the category of right~$\bfG$-spaces. By Theorem~\ref{thm:centers_of_sPre} it follows that the homotopy coherent center of this category is just the homotopy coherent center of~$\bfG$ itself. The latter~(and its relation to the strict center~$\Z(\bfG)$ as studied in group theory) is already discussed in~\cite[Section~5]{Szymik}. 

Note that the homotopy theory encoded in the simplicial Quillen model category structure on~$\sPre(\bfG)$ is the one where the equivalences are the~$\bfG$-maps that are equivalences as maps of (underlying) spaces. In particular, cofibrant objects are free. The reader who desires to understand the homotopy coherent center of the category of~$\bfG$-spaces, where the homotopy theory takes fixed point data into account, will have no problems of modeling the homotopy theory as diagrams over a suitable orbit category. Then Theorem~\ref{thm:centers_of_sPre} applies as well.
\end{example}

\begin{example}\label{ex:overB}
Let~$B$ be a Kan complex, and consider the slice category~$\bfS\downarrow B$ of spaces over~$B$. The Quillen model category structure on~$\bfS\downarrow B$ is the one in which the equivalences and fibrations are created by the forgetful functor~$\bfS\downarrow B\to\bfS$. The result is Quillen equivalent to the category of~$\bfG B$-spaces, where~$\bfG B$ is the loop groupoid of~$B$, a simplicial groupoid that models the loop spaces of~$B$. By Theorem~\ref{thm:centers_of_sPre} it follows that the homotopy coherent center of this category is just the homotopy coherent center of~$\bfG B$ itself. The latter is already discussed in~\cite[Section~8]{Szymik}, and the result is
\begin{equation}\label{eq:parallel_transport}
\hoZ((\bfS\downarrow B)^\fc)\simeq\Omega(\Map(B,B),\id_B).
\end{equation}
Note that this can also be written
\[
\Omega(\Aut(B),\id_B)\simeq\Omega^2\rmB\!\Aut(B),
\]
since the component of the identity consists of homotopy automorphisms (equivalences). 

Actually, the simplicial set~$B$ does not have to be Kan complex. We can always replace is by an equivalent Kan complex~$B'$. Then the categories~$\bfS\downarrow B$ and~$\bfS\downarrow B'$ are Quillen equivalent, and the conclusion holds true if~$\Map(B,B)$ is replaced by the derived mapping space~$\Map(B',B')$.

We can offer the following interpretation of the equivalence~\eqref{eq:parallel_transport}. An element~$\omega$ in the space on the right hand side gives, for each point~$b$ of~$B$, a loop~$\omega(b)$ based at~$b$. A fibrant object in~$\bfS\downarrow B$ is a fibration~$E\to B$, and~$\omega$ corresponds to the map~$E\to E$ (over~$B$) that is the parallel transport~$E(b)\to E(b)$ along~$\omega(b)$ in the fibre~$E(b)$ over~$b$.
\end{example}


\section{Localizations}\label{sec:localizations}

We have seen in the previous Section~\ref{sec:presheaves} what happens to the homotopy coherent centers if we pass from a simplicial category~$\bfC$ to its category~$\sPre(\bfC)$ of simplicial presheaves. In this section, we study the passage from a Quillen model category~$\bfM$~(that may or may not be of the form~\hbox{$\bfM=\sPre(\bfC)$}) to one of its left Bousfield localizations. As we will see, there are examples where the homotopy coherent center is changed (Example~\ref{ex:center_changed}), but there are also conditions that ensure that this does not happen (Theorem~\ref{thm:when_localization_preserves_centers}). We start with a review of localizations, if only to fix notation.


\subsection{Bousfield localizations}

Let us first recall the essential aspects of Bousfield localizations in the general context of simplicial categories before we specialize to the examples of interest that all come from Quillen model categories.

\begin{definition}
Let~$\bfB$ be a simplicial category. A {\em left Bousfield localization} of~$\bfB$ is a simplicial category~$\bfL$ together with a homotopy adjunction (Definition~\ref{def:homotopy_adjunction})
\begin{equation}\label{eq:lBl}
I\colon\bfB\longleftrightarrow\bfL\colon J,
\end{equation}
with~$I$ the left adjoint, such that the right adjoint~$J$ is homotopically full and faithful~(Definition~\ref{def:homotopically_ff}). 
\end{definition}

It is an immediate consequence of the definition that the counits~$IJy\to y$ are equivalences in~$\bfL$, since there are equivalences
\[
\bfL(IJy,t)\simeq\bfB(Jy,Jt)\simeq\bfL(y,t)
\]
for all test objects~$t$. The other composition~$JI\colon\bfB\to\bfB$ is usually denoted by~$L$, and is called the associated {\em localization functor}. It is (split) homotopically idempotent, so that there exists an equivalence~$L^2\simeq L$, because the counit is an equivalence. The unit~$\Id_\bfB\to L$ induces equivalences
\[
\bfB(Lx,Ly)\overset{\sim}{\longrightarrow}\bfB(x,Ly)
\]
for all objects~$x$ and~$y$. An object~$x$ on~$\bfB$ is {\em local} if the unit~$x\to JIx=Lx$ of the adjunction is an equivalence in~$\bfB$. This is the case if and only if~$x\simeq Ly$ for some object~$y$ of~$\bfB$. 


\subsection{Localizing Quillen model categories}

Left Bousfield localizations arise from localizations~$W^{-1}\bfM$ of Quillen model categories~$\bfM$ with respect to a set~$W$ of morphisms in~$\bfM$ that one would like to add to the weak equivalences.

In our applications, the set $W$ (or the subcategory~$\bfW\subseteq\bfM$ that it generates) is always determined by a morphism~$f\colon A\to B$ in~$\bfM$ in the following sense. An object~$Z$ of~$\bfM$ is called {\em$f$-local} if and only if the map
\[
\bfM(f^\fc,Z^\fc)\colon\colon\bfM(B^\fc,Z^\fc)\to\bfM(A^\fc,Z^\fc)
\]
induced by the morphism~$f$ is an equivalence. A morphism~$g\colon X\to Y$ is called an~{\em$f$-equivalence} if and only if the maps
\[
\bfM(g^\fc,Z^\fc)\colon\bfM(Y^\fc,Z^\fc)\to\bfM(X^\fc,Z^\fc)
\]
are equivalences for all~$f$-local objects~$Z$. For example, it is immediate that the morphism~$f$ itself is an~$f$-equivalence. The~$f$-equivalences form a subcategory~$\bfW_f\subseteq\bfM$ that has the same objects, since all identities are~$f$-equivalences. All examples of pairs~$(\bfM,\bfW)$ considered in this text can be written in the form~$(\bfM,\bfW_f)$ for a suitable morphism~$f$.

Given a morphism~$f$ in a (simplicial) category~$\bfP=\sPre(\bfC)$ of simplicial presheaves, there are actually simplicial Quillen model structures on~$\bfP$ such that the class of model equivalences agrees with the class~$\bfW_f$ of~$f$-equivalences as described before. In one such, the cofibrations are formed objectwise, and then the fibrant objects are the~$f$-local objects that are fibrant as simplicial presheaves. See~\cite[Theorem~4.1.1]{Hirschhorn}, for example. Hirschhorn uses the technical conditions of left properness and cellularity on~$\bfP$. But, he shows in~\cite[Proposition~4.1.4]{Hirschhorn} and~\cite[Proposition~4.1.7]{Hirschhorn} that these are satisfied for~$\bfP=\sPre(\bfC)$.


\subsection{Density}

The right adjoint of a left Bousfield localization is always full and faithful. This and the following result will be the first step towards the computation of homotopy coherent centers of localizations. 

\begin{proposition}{\upshape\bf(Density in localizations)}\label{prop:localization_dense}
For all left Bousfield localizations there exists an equivalence
\[
\Id_\bfL\simeq\hoLKan_II
\]
as functors~$\bfL\to\bfL$, so that the functor~$I$ is homotopically dense in the sense of Definition~\ref{def:homotopical_dense}.
\end{proposition}

\begin{proof}
The homotopy left Kan extension is define object-wise, so that the value on an object~$y$ in~$\bfL$ is the homotopy colimit of the diagram
\[
(x,Ix\to y)\longmapsto Ix
\]
in~$\bfL$, where the source is the slice category~$I\downarrow y$ of~$I$ over~$y$. Since the counit~\hbox{$IJy\to y$} is an equivalence in~$\bfL$, this category has a homotopically terminal object, namely the pair~\hbox{$(Jy,IJy\to y)$}. It follows that there exists an equivalence
\[
\hoLKan_II\simeq IJ
\]
of functors. Using again that the counit~$IJ\to\Id_\bfL$ is an equivalence of functors, the result follows.
\end{proof}


\subsection{Centers of localizations}

The proof of our main result will rely on the fact that we can sometimes pass from a category~$\sPre(\bfC)$ of simplicial presheaves to one of its~$f$-localizations without changing the centers. This is not the case in general, as the following example shows.

\begin{example}\label{ex:center_changed}
If the~$f$-local objects in~$\sPre(\bfC)$ are those that send every arrow in~$\bfC$ to an equivalence, then~$\bfC^{-1}\sPre(\bfC)$ is equivalent to the category of bundles over the nerve~$X=\rmN\bfC$, see~\cite[4.9]{Dwyer:Localizations}, so that this class of examples is essentially a special case of Example~\ref{ex:overB}. All objects are automatically~$f$-local when~$\bfC$ is a groupoid, but we are really interested in the case when~$\bfC$ is not a groupoid. For such an~$f$, the homotopy coherent center of~$\bfC^{-1}\sPre(\bfC)^\fc$ is equivalent to the space of sections of the evaluation fibration~$\Lambda X\to X$ from the free loop space~$\Lambda X=\Map(\rmS^1,X)$ of the nerve~$X=\rmN\bfC$.~(Compare with the proof of Proposition~5.2 in~\cite{Szymik}.) This space is discrete if~$\bfC$ is a discrete groupoid, but not for general discrete categories~$\bfC$. In fact, by McDuff's improvement~\cite{McDuff} of the Kan-Thurston theorem~\cite{Kan+Thurston}, every connected space is equivalent to the nerve of some discrete monoid.~(The note~\cite{Fiedorowicz} contains an example of a monoid with five elements such that its nerve is equivalent to the~$2$-sphere.) In contrast, the homotopy coherent center of~$\sPre(\bfC)^\fc$ is homotopically discrete for all discrete categories~$\bfC$ by Theorem~\ref{thm:centers_of_sPre}.
\end{example}

Keeping this warning in mind, we start our study of homotopy coherent centers of localizations with some observations that are true in general.

\begin{theorem}\label{thm:Bousfield_general}
For every left Bousfield localization~$I\colon\bfB\leftrightarrow\bfL\colon J$, there exists an equivalence
\[
\hoZ(\bfL)\simeq\hoNat(\Id_\bfB,L)
\]
where~$L=JI\colon\bfB\to\bfB$ is again the localization functor.
\end{theorem}

\begin{proof}
By Corollary~\ref{cor:dense_and_centers}, the homotopical density in localizations (Proposition~\ref{prop:localization_dense}) has the consequence that we can always express the homotopy coherent center of~$\bfL$ as
\[
\hoZ(\bfL)=\hoNat(\Id_\bfL,\Id_\bfL)\simeq\hoNat(I,I).
\]
The adjunction
\[
\hoNat(I,I)\simeq\hoNat(\Id_\bfB,JI)
\]
and the definition~$L=JI$ then prove the result.
\end{proof}

An inspection of the definition of the space of homotopy coherent natural transformations in Section~\ref{subsec:N1} indicates the usefulness of this observation: The mapping spaces of the localized category~$\bfL$ occur everywhere in the term
\[
\prod_{x_0,\dots,x_n}\Map(\bfL(x_1,x_0)\times\dots\times\bfL(x_n,x_{n-1}) , \bfL(x_n,x_0) )
\]
involved in the homotopy coherent center of~$\bfL$, but the result moves us back to within~$\bfB$, 
where we only have to deal with
\[
\prod_{x_0,\dots,x_n}\Map(\bfB(x_1,x_0)\times\dots\times\bfB(x_n,x_{n-1}) , \bfB(x_n,Lx_0) ),
\]
for the evident price that we have to replace one identity functor by the localization functor.


\subsection{Localization of categories of presheaves}

The following result gives a criterion when the localization of a category of simplicial presheaves does not change the homotopy coherent center.

\begin{theorem}\label{thm:when_localization_preserves_centers}
Assume that~$I\colon\bfB\leftrightarrow\bfL\colon J$ is a left Bousfield localization of a category~\hbox{$\bfB=\sPre(\bfC)$} of simplicial presheaves such that the representable presheaves are local. Then there are equivalences
\[
\hoZ(\bfL)\simeq\hoZ(\bfB)
\]
of homotopy coherent centers.
\end{theorem}

\begin{proof}
We begin by applying Theorem~\ref{thm:Bousfield_general} to see that there exists an equivalence
\[
\hoZ(\bfL)\simeq\hoNat(\Id_\bfB,L).
\]
We can apply Proposition~\ref{prop:homotopically_dense} (about homotopically dense functors) in the special situation~\hbox{$F=\rmY_\bfC$}~(the Yoneda embedding) and~$G=L$. This gives
\[
\hoNat(\Id_\bfB,L)\simeq\hoNat(\rmY_\bfC,L\rmY_\bfC).
\]
By hypothesis, the representable objects are local, so that the unit~$\rmY_\bfC\to L\rmY_\bfC$ is an equivalence of functors. Thus, we have an equivalence
\[
\hoNat(\rmY_\bfC,L\rmY_\bfC)\simeq\hoNat(\rmY_\bfC,\rmY_\bfC),
\]
and the right hand side has already been identified with~$\hoZ(\bfB)$ by Theorem~\ref{thm:centers_of_sPre} and our proof of it.
\end{proof}

Of course, Theorem~\ref{thm:centers_of_sPre} also shows that both of the homotopy coherent centers in the preceding result are equivalent to~$\hoZ(\bfC)$. Again, we see in particular that the result is essentially small.


\subsection{Applications}\label{ssec:apps_for_loc}

Before we give a more substantial example, let us state an immediate consequence of Theorem~\ref{thm:when_localization_preserves_centers}.

\begin{example}
We can generalize Example~\ref{ex:spaces}, where we have shown that the homotopy coherent center of the category of spaces is trivial: In fact, the homotopy coherent center of any left Bousfield localization of the category of spaces is trivial. According to the theorem, we have to check that the (only) representable object~$\star$ is local. But the identity~$\star=\star$ is a fibration in any Quillen model structure on the category of spaces.
\end{example}

We have already mentioned in Remark~\ref{rem:sPre_model} that the category of simplicial presheaves on a (discrete) category~$\bfC$ has a `global' homotopy theory. If~$\bfC$ comes with a Grothendieck topology~$\tau$, so that~$(\bfC,\tau)$ is a Grothendieck site, then the global homotopy theory of simplicial presheaves can be localized with respect to a suitable morphism~$f=f(\tau)$ so as to obtain a~$\tau$-local homotopy theory, where an equivalence is a morphism that induces an isomorphism on all~$\tau$-sheaves of homotopy groups. The local objects are the simplicial~$\tau$-sheaves. See~\cite{Dugger+Hollander+Isaksen} for details. This raises the question of the homotopy coherent center of the category of simplicial~$\tau$-sheaves, and we can answer it for the class of Grothendieck topologies that appears most often in practice.

Recall that a Grothendieck topology~$\tau$ on~$\bfC$ is called {\em subcanonical} whenever all representable presheaves are~$\tau$-sheaves. 

\begin{example}
From their very definition, schemes over a base scheme~$S$ represent sheaves in the Zariski topology. By Grothendieck's descent theorems, they are even sheaves for the fpqc topology, and consequently for every coarser Grothendieck topology such as the fppf, the \'etale, and the Nisnevich topologies.
\end{example}

Theorem~\ref{thm:when_localization_preserves_centers} applies to localizations of categories of simplicial presheaves with respect to subcanonical Grothendieck topologies.

\begin{theorem}
For every subcanonical Grothendieck topology~$\tau$ on~$\bfC$, the homotopy coherent center of the category of simplicial~$\tau$-sheaves is equivalent to the homotopy coherent center of the category of simplicial presheaves on~$\bfC$.
\end{theorem}

Of course, both of the homotopy coherent centers are equivalent to~$\hoZ(\bfC)$ by Theorem~\ref{thm:centers_of_sPre}, and essentially small.


\section{Algebraic theories}\label{sec:theories}

We will choose a skeleton~$\Sigma$ of the category of finite sets. For each integer~$n\geqslant0$ such a category has a unique object~$\sigma_n$ with precisely~$n$ elements, and in~$\Sigma$ the object~$\sigma_{m+n}$ is the coproduct~(sum) of the objects~$\sigma_m$ and~$\sigma_n$. 

An {\em algebraic theory}~$\Theta$ is a (discrete) category together with a functor~$\Sigma\to\Theta$ that preserves coproducts and is bijective on objects. The image of~$\sigma_n$ will be written~$\theta_n$.

We remark that some authors prefer to work with~$\Theta^\op$, so that~$\theta_n$ is the product of~$n$ copies of~$\theta_1$. For example, this is Lawvere's convention when he introduced this notion in~\cite{Lawvere}, and it is also used in the monograph~\cite{Adamek+Rosicky+Vitale}. Our convention reflects the point of view that~$\theta_n$ should be thought of as the free~$\Theta$-algebra on~$n$ generators, covariantly in~$n$. On the other hand, the opposite category~$\Theta^\op$ is hidden in the presheaf speak to be employed below.


\subsection{Discrete algebras}

Let~$\Theta$ be a (discrete) algebraic theory. A {\em$\Theta$-algebra} is a presheaf~$A$~(of sets) that sends coproducts in~$\Theta$ to products of sets. This defines a full subcategory~\hbox{$\Alg(\Theta)\subseteq\Pre(\Theta)$}. The image of the Yoneda embedding
\[
\rmY_\Theta\colon\Theta\longrightarrow\Pre(\Theta),\,\theta_n\mapsto\Theta(?,\theta_n)
\]
lies in~$\Alg(\Theta)$. In this way, the algebraic theory~$\Theta$ can indeed be thought of as a full subcategory of (finitely generated and free)~$\Theta$-algebras.

Categories of the form~$\Alg(\Theta)$ can be also be characterized as the categories of algebras over monads (triples) that preserve filtered colimits, see~\cite[Appendix~A]{Adamek+Rosicky+Vitale}. 


\subsection{Simplicial algebras}\label{sec:strict_algebras}

Let~$\Theta$ be a (discrete) algebraic theory, and we are interested in the category of simplicial~$\Theta$-algebras. This is the full subcategory~\hbox{$\sAlg(\Theta)\subseteq\sPre(\Theta)$} of simplicial presheaves~$A\colon\Theta^\op\to\bfS$ that send coproducts to products. Just as for presheaves on a discrete category, a simplicial~$\Theta$-algebra~$A$ can be thought of both as an algebraic structure of type~$\Theta$ on the space~$A(\theta_1)$, or as a simplicial object in the category of~(discrete)~$\Theta$-algebras. It follows that~$\sAlg(\Theta)$ is indeed the category of simplicial objects in~$\Alg(\Theta)$.

\begin{remark}
The category~$\sAlg(\Theta)$ of simplicial~$\Theta$-algebras comes with simplicial Quillen model structures such that the equivalences are formed objectwise. In one such, due to Reedy, the fibrations are also formed objectwise. See~\cite[Theorem~3.1]{Schwede}, for example.
\end{remark}

In this section, we will determine the centers of the simplicial categories~$\sAlg(\Theta)$. 


\subsection{Homotopy algebras}\label{sec:homotopy_algebras}

More amenable to homotopy considerations than the simplicial~$\Theta$-algebras are the {\em homotopy~$\Theta$-algebras}. These are the simplicial presheaves that send coproducts to products only up to homotopy. Let us recall from~\cite{Badzioch} how this idea can be made precise using localizations with respect to a morphism~$f$. This approach will also bring the relevant homotopy theory along.

Since~$\theta_n$ is the coproduct of~$n$ copies of~$\theta_1$, there are~$n$ maps~$\theta_1\to\theta_n$ that induce an isomorphism
\[
n\cdot\theta_1=\underbrace{\theta_1+\dots+\theta_1}_n\longrightarrow\theta_n.
\]
The Yoneda embedding~$\rmY_\Theta$ preserves products, but not necessarily coproducts. In other words, the natural morphism
\begin{equation}\label{eq:at_inversion}
f_n\colon n\cdot\rmY_{\theta_1}\longrightarrow\rmY_{n\cdot\theta_1}\cong\rmY_{\theta_n}.
\end{equation}
need not be an isomorphism. We set~$f$ to be the coproduct of these maps (over all~$n$). This map~$f$ has the property that a simplicial presheaf~$A$ is a simplicial~$\Theta$-algebra~(maps coproducts to products) if and only if~$\sPre(\Theta)(f,A)$ is an isomorphism. 

We can now localize with respect to the set~$F=\{f_n\}$ of maps~\eqref{eq:at_inversion}. The category~$F^{-1}\sPre(\Theta)$ of simplicial presheaves has simplicial Quillen model structures such that the equivalences are the~$f$-equivalences. This is a consequence of the general theory recalled in Section~\ref{sec:localizations} here, and in this particular case it is due to Badzioch, see~\cite[Proposition~5.4]{Badzioch}. The fibrant objects are the~$f$-local objects, which are precisely the homotopy~$\Theta$-algebras on Kan complexes by the same result. This shows that the following definition captures the essence of homotopy~$\Theta$-algebras from our point of view.

\begin{definition}
Let~$\Theta$ be an algebraic theory. The category~$\hAlg(\Theta)$ is the localization 
\[
\hAlg(\Theta)=F^{-1}\sPre(\Theta)
\]
of the category of simplicial presheaves with respect to the set~$F$ of morphisms~$f_n$ as in~\eqref{eq:at_inversion}. Its fibrant objects will be referred to as {\em homotopy~$\Theta$-algebras}.
\end{definition}

As a consequence of the definition, we have a sequence
\[
\sAlg(\Theta)\longrightarrow
\sPre(\Theta)\longrightarrow
\hAlg(\Theta)
\]
of simplicial functors, and an equality
\begin{equation}\label{eq:localized_model_for_hAlg}
\hoZ(\hAlg(\Theta)^\fc)=\hoZ(F^{-1}\sPre(\Theta)^\fc)
\end{equation}
of homotopy coherent centers.


\subsection{Replacing strict algebras by homotopy algebras}

For the sake of calculations, it does not matter whether we work with strict or homotopy~$\Theta$-algebras:

\begin{proposition}\label{prop:s=h}
The inclusion~$\sAlg(\Theta)\to\hAlg(\Theta)$ induces an equivalence
\[
\hoZ(\sAlg(\Theta)^\fc)\simeq\hoZ(\hAlg(\Theta)^\fc)
\]
of homotopy coherent centers.
\end{proposition}

\begin{proof}
Badzioch has shown in~\cite[Theorem~6.4]{Badzioch} as the main result of that paper that the inclusion functor from the simplicial Quillen model category of~$\Theta$-algebras~$\sAlg(\Theta)$ to~$\hAlg(\Theta)$ is a right Quillen equivalence with respect to the Quillen model structures that we have described in Sections~\ref{sec:strict_algebras} and~\ref{sec:homotopy_algebras}.
\end{proof}


\subsection{Homotopy coherent centers of categories of algebras}

We are now ready to prove a main result of this text.

\begin{theorem}\label{thm:centers_of_sAlg}
For every algebraic theory~$\Theta$, 
the inclusion of the category of finitely generated free~$\Theta$-algebras induces an equivalence between the homotopy coherent center~$\hoZ(\sAlg(\Theta)^\fc)$ of the category of simplicial~$\Theta$-algebras and~$\hoZ(\Theta)$.
\end{theorem} 

\begin{proof}
The crucial observation is that representable presheaves~$\rmY_\bfC(x)$ send coproducts to products. This implies for us that the presheaves~$\rmY_\Theta(\theta_n)$ are~$F$-local for the set~$F$ such that~\hbox{$F^{-1}\sPre(\Theta)=\hAlg(\Theta)$}. (Loosely speaking: Free~$\Theta$-algebras are homotopy~$\Theta$-algebras.) Thus, Theorem~\ref{thm:when_localization_preserves_centers} applies and allows us to deduce
\[
\hoZ(\sPre(\Theta)^\fc)\simeq\hoZ(F^{-1}\sPre(\Theta)^\fc).
\]
By Theorem~\ref{thm:centers_of_sPre}, the left hand side is equivalent to~$\hoZ(\Theta)$, and by~\eqref{eq:localized_model_for_hAlg}, the right hand side is~$\hoZ(\hAlg(\Theta)^\fc)$. Finally, Proposition~\ref{prop:s=h} implies that we can add~$\hoZ(\sAlg(\Theta)^\fc)$ to this chain of equivalences as well. 
\end{proof}

\begin{remark}
If the theory~$\Theta$ at hand is discrete, its homotopy coherent center is just the ordinary center~$\hoZ(\Theta)\simeq\Z(\Theta)$, and the same holds for the category of discrete~$\Theta$-algebras: There is an equivalence~$\hoZ(\Alg(\Theta))\simeq\Z(\Alg(\Theta))$. We are left with the task to compare the ordinary centers, which is now elementary: Since the inclusion~\hbox{$\Alg(\Theta)\subseteq\Pre(\Theta)$} is full and faithful, the inclusion~$\Theta\subseteq\Alg(\Theta)$ is both dense and full and faithful. Hence, it induces an isomorphism~$\Z(\Theta)\cong\Z(\Alg(\Theta))$.
\end{remark}


\subsection{Applications}

The following example has already been mentioned in the introduction.

\begin{example}
The homotopy coherent center
\[
\hoZ(\text{simplicial commutative}~\bbF_p\text{-algebras})\simeq\bbN
\]
is homotopically discrete, generated by Frobenius.
\end{example}

Here is another example that is even more fundamental.

\begin{example}\label{ex:pointed}
Let~$\Gamma^\op$ be the category of finite pointed sets and pointed maps. This seems like a very unadventurous theory: There are no compositions, and there is only one distinguished constant. In other words, the category~$\sAlg(\Gamma^\op)$ is the category~\hbox{$\star\downarrow\bfS$} of pointed spaces. The homotopy coherent center
\[
\hoZ(\star\downarrow\bfS)\simeq\{0,1\}
\]
is equivalent to the monoid~$\{0,1\}$ under multiplication. The two elements are realized by the family of constant self-maps and by the family of identities, respectively.
\end{example}

While the statement of Theorem~\ref{thm:centers_of_sAlg} concerns the strict algebras, the proof uses that their homotopy theory is equivalent to that of the homotopy algebras. Since the latter are known to model some very interesting homotopy theories in their own right, let us spell this out in some more examples.

\begin{example}
Loop spaces are essentially homotopy algebras over the theory of groups. The center of the category of~(free) groups is again the monoid~$\{0,1\}$ under multiplication. This gives
\[
\hoZ(\text{$\Omega$-spaces})\simeq\{0,1\}.
\]
Note that the homotopy theory of loop spaces is equivalent to the homotopy theory of pointed~$0$-connected spaces.
\end{example}

\begin{example}
Similarly, homotopy monoids model all~$\rmA_\infty$-spaces. The center of the category of (free) monoids is again~$\{0,1\}$, and this gives
\[
\hoZ(\rmA_\infty\text{-spaces})\simeq\{0,1\}.
\]
\end{example}

It would be interesting to extend these results to~$\Omega^n$-spaces and~$\rmE_n$-spaces for all higher iterations~\hbox{$2\leqslant n \leqslant\infty$}.


\section*{Acknowledgment}

The second author has been supported by the Danish National Research Foundation through the Centre for Symmetry and Deformation (DNRF92).



\vfill

William G. Dwyer\\
Department of Mathematics\\ 
University of Notre Dame\\ 
Notre Dame, IN 46556\\
USA\\
\href{mailto:dwyer.1@nd.edu}{dwyer.1@nd.edu}

Markus Szymik\\
Department of Mathematical Sciences\\
NTNU Norwegian University of Science and Technology\\
7491 Trondheim\\
NORWAY\\
\href{mailto:markus.szymik@math.ntnu.no}{markus.szymik@math.ntnu.no}


\end{document}